\documentclass{siamart1116} 

\usepackage{algorithmicx}
\usepackage{amsmath}
\usepackage{amssymb}
\usepackage{algorithm}
\usepackage{algpseudocode}
\usepackage[title]{appendix}
\usepackage[english]{babel}
\usepackage{booktabs}
\usepackage{graphicx}
\usepackage{hyperref}
\usepackage[utf8]{inputenc}
\usepackage{mathtools}
\usepackage{multirow}
\usepackage{siunitx}
\usepackage{tcolorbox}
\usepackage{varwidth}
\usepackage{verbatim}
\usepackage{geometry}
\usepackage{url}
\usepackage{wrapfig}

\numberwithin{theorem}{section}
\newcommand{\TheTitle}{Simplified energy landscape for modularity using total variation}
\newcommand{\TheAuthors}{Z. Boyd, E. Bae, X.-C. Tai, and A. L. Bertozzi}
\newcommand{\TheShortTitle}{Simplified landscape for modularity using TV}
\headers{\TheShortTitle}{\TheAuthors} 
\title{{\TheTitle}\thanks{Submitted to the editors DATE.}}
\author{\TheAuthors}
\author{
	Zachary M. Boyd\footnote{D\MakeLowercase{epartment} \MakeLowercase{of} M\MakeLowercase{athematics}, UCLA, L\MakeLowercase{os} A\MakeLowercase{ngeles}, CA (\MakeLowercase{\email{zach.boyd@math.ucla.edu})}.}
	\and
	Egil Bae\footnote{N\MakeLowercase{orwegian} D\MakeLowercase{efense} R\MakeLowercase{esearch} E\MakeLowercase{stablishment} (FFI), K\MakeLowercase{jeller}, N\MakeLowercase{orway} (\MakeLowercase{\email{Egil.Bae@ffi.no})}.}
	\and
	Xue-Cheng Tai\footnote{D\MakeLowercase{epartment} \MakeLowercase{of} M\MakeLowercase{athematics}, U\MakeLowercase{niversity} \MakeLowercase{of} B\MakeLowercase{ergen}} \footnote{D\MakeLowercase{epartment} \MakeLowercase{of} M\MakeLowercase{athematics}, H\MakeLowercase{ong} K\MakeLowercase{ong} B\MakeLowercase{aptist} U\MakeLowercase{niversity}, K\MakeLowercase{owloon} T\MakeLowercase{ong} K\MakeLowercase{owloon}, H\MakeLowercase{ong} K\MakeLowercase{ong} (\MakeLowercase{\email{tai@math.uib.no})}}
	\and
	Andrea Bertozzi\footnote{D\MakeLowercase{epartment} \MakeLowercase{of} M\MakeLowercase{athematics}, UCLA, L\MakeLowercase{os} A\MakeLowercase{ngeles}, CA (\MakeLowercase{\email{bertozzi@math.ucla.edu})}.}
}

\newcommand{\R}{\mathbb{R}}
\newcommand{\N}{\mathbb{N}}
\newcommand{\F}{\mathcal{F}}

\DeclareMathOperator{\threshold}{threshold}
\DeclareMathOperator{\vol}{vol}
\DeclareMathOperator{\const}{const}
\DeclareMathOperator{\nhat}{\hat{n}}
\DeclareMathOperator{\Cut}{Cut}
\DeclareMathOperator{\mean}{mean}

\DeclareMathOperator{\argmin}{argmin}
\DeclareMathOperator{\argmax}{argmax}
\DeclareMathOperator{\Div}{div}

\DeclareMathOperator{\grad}{\nabla}
\DeclareMathOperator{\diag}{diag}

\graphicspath{{./images/}}

\allowdisplaybreaks 

\PassOptionsToPackage{hyphens}{url}\usepackage{hyperref}

\begin{document}

\maketitle

\begin{abstract}
	Networks capture pairwise interactions between entities and are frequently used in applications such as social networks, food networks, and protein interaction networks, to name a few. Communities, cohesive groups of nodes, often form in these applications, and identifying them gives insight into the overall organization of the network. One common quality function used to identify community structure is \emph{modularity}. In Hu et al.~[SIAM J.\ App.\ Math., 73(6), 2013], it was shown that modularity optimization is equivalent to minimizing a particular nonconvex total variation (TV) based functional over a discrete domain. They solve this problem---assuming the number of communities is known---using a Merriman, Bence, Osher (MBO) scheme.

	We show that modularity optimization is equivalent to minimizing a \emph{convex} TV-based functional over a discrete domain---again, assuming the number of communities is known. Furthermore, we show that modularity has no convex relaxation satisfying certain natural conditions. We therefore, find a manageable non-convex approximation using a Ginzburg Landau functional, which provably converges to the correct energy in the limit of a certain parameter. We then derive an MBO algorithm with fewer hand-tuned parameters than in Hu et al.\ and which is 7 times faster at solving the associated diffusion equation due to the fact that the underlying discretization is unconditionally stable. Our numerical tests include a hyperspectral video whose associated graph has $2.9\times 10^7$ edges, which is roughly 37 times larger than was handled in the paper of Hu et al.
\end{abstract}



\section{Introduction}

Community detection in complex networks is a difficult problem with applications in numerous disciplines, including social network analysis~\cite{multiplex}, molecular biology~\cite{hartwell_et_al}, politics~\cite{multiplex}, material science~\cite{granular}, and many more~\cite{newman_book}. There is a large and growing literature on the subject, with many competing definitions of community and associated algorithms \cite{fortunato_2010, porter2009survey, fortunato_hric}. In practice, community detection is used as a way to understand the coarse, or mesoscale, properties of networks. Further investigation into these communities sometimes leads to insights about the processes that formed the network or the dynamics of processes acting on the network.

In this paper, we focus on the task of partitioning the nodes in a complex network into disjoint communities, although many other variations, such as overlapping, fuzzy, and time-dependent communities are also used in the literature. The proper way to understand such communities in small networks has been fairly well-studied, and their role in larger networks is the subject of active research~\cite{think_locally}.

A great variety of definitions have been proposed to make the partitioning task precise~\cite{fortunato_hric}, including notions involving edge-counting, random walk trapping, information theory, and---especially recently---generative models such as stochastic block models (SBMs). In this paper, we focus on modularity optimization~\cite{newman_girvan_2004}, which is the most well-studied of existing methods. To define it, we need the following terminology, which is used throughout the paper.
\begin{definition}
	Let $G$ be a non-negatively weighted, undirected, sparse graph with $N$ nodes, weight matrix $W = (w_{ij})$, degree vector $k$ satisfying $k_i = \sum_j w_{ij},$ and $2m = \sum_i k_i$.
	\label{def:G}
\end{definition}
Modularity-optimizing algorithms seek a partition $A_1,\ldots,A_{\nhat}$ of the nodes of $G$ which maximizes
\[ Q = \frac{1}{2m}\sum_{\ell = 1}^{\nhat} \sum_{ij \in A_{\ell}} w_{ij} - \frac{k_ik_j}{2m}. \]
Intuitively, we are to understand $w_{ij}$ as the observed edge weight and $\frac{k_i k_j}{2m}$ as the expected weight if the edges had been placed at random. Thus, there is an incentive to group those nodes which have an unusually strong connection under the null model.

The results of modularity optimization must be interpreted carefully. For example, the modularity functional, $Q$, will find communities in a random graph~\cite{guimere_pardo_nunes_amaral_2004}. In addition, many dissimilar partitions may yield near-optimal modularity values~\cite{clauset_2010}. This is to be expected, since the network partitioning problem is very well posed. Real networks are generated by complicated processes with many factors, and thus there are often multiple ways to partition a network that reflect legitimate divisions among the objects being studied~\cite{clauset_2016}. One way to leverage this diversity of high-modularity partitions in practice, as well as prevent the discovery of communities in random graphs, relies on consensus clustering~\cite{zhang_moore_2016}. Another approach is simply to sample many high-modularity partitions, expecting that multiple intuitively-meaningful partitions may be found. Such effects have been observed, for instance, in the Zachary Karate Club network, which has both a community structure and leader-follower structure~\cite{clauset_2016}.


Modularity also has preferred scale for communities~\cite{fortunato_barthelemy_2007,lancichinetti_fortunato_2011}. For this reason, one typically includes a resolution parameter $\gamma > 0$~\cite{reichardt_bornholdt_2006,arenas_resolution_2008}, yielding
\[ Q = \frac{1}{2m}\sum_{\ell = 1}^{\nhat} \sum_{ij \in A_{\ell}} w_{ij} - \gamma\frac{k_ik_j}{2m} \]
When $\gamma$ is nearly zero, the incentive is to place many nodes in the same community, so that the edge weight is included in the sum. When $\gamma$ is large, few nodes are placed in each community, to avoid including the large penalty term $\gamma \frac{k_ik_j}{2m}$.

A number of heuristics have been proposed to optimize modularity~\cite{fortunato_2010,fortunato_hric}, with prominent approaches including spectral~\cite{newman_spectral,newman_spectral_2}, simulated annealing~\cite{guimere_pardo_nunes_amaral_2004}, and greedy or Louvain algorithms~\cite{blondel_2008}. It can also be interpreted in terms of force-directed layout and optimized using visualization techniques~\cite{noack}. The modularity optimization problem is NP-hard~\cite{bandes_et_al_2008}, so it is not expected that a single heuristic will suffice for all situations. 

In 2013, Hu, Laurent, Porter, and Bertozzi~\cite{hui} discovered a connection between the modularity optimization problem in network science and total variation (TV) minimization from image processing. As an application, Hu et al.\ developed Modularity MBO, a TV-oriented optimization algorithm that effectively optimizes modularity. The present work strengthens both theoretical and algorithmic connections from~\cite{hui}. Specifically, we make the following contributions:

We start with derivations of four formulations of modularity, two in terms of TV and two in terms of graph cuts, which inspire the subsequent analysis. In addition to being intuitively simple, these formulas place all of the nonconvexity of the problem into a discrete constraint---the functionals themselves are convex. We prove a theorem showing that convex relaxation of modularity is not possible under certain conditions. While many practitioners have observed that modularity optimization seems highly nonconvex, ours is the first result of which we are aware showing this in a rigorous way. We then provide an alternative relaxation, using the Ginzburg-Landau functional, that smooths the discrete constraint so that it becomes manageable. We end the theory section by showing that solutions of our relaxed problem converge to maximizers of modularity in an appropriate sense.

Based on these ideas, and following~\cite{hui}, we develop an MBO-type scheme, Balanced TV, which quickly and accurately optimizes modularity in several examples. This algorithm seems especially well-suited to similarity networks from machine learning, where prior knowledge of the number of communities is available and the number of such communities tends to be modest. Using the convexity of our formulation of TV, we provide inner- and outer-loop timestep bounds to avoid hand-tuning parameters, as is necessary in~\cite{hui}. We also show how to discretize the partial differential equation (PDE) part of the MBO iteration in an unconditionally stable, efficient way. We test our algorithm on much larger datasets than are used in~\cite{hui}. Finally, we show that this approach can solve semi-supervised problems as well.

The rest of the paper is organized as follows: \Cref{sec:background} surveys the necessary background in both modularity optimization and TV minimization. \Cref{sec:theorem} develops the main theoretical results about the optimization problem itself. \Cref{sec:numerics} develops the theory and practical implementation of our algorithm, Balanced TV. \Cref{sec:results} gives numerical examples. \Cref{sec:conclusion} concludes. There are also appendices containing additional background and deferred proofs.

\section{Total Variation Optimization: Continuum and Discrete}
\label{sec:background}

While modularity optimization is normally understood as a combinatorial problem, TV was historically seen as a continuum object, with applications in partial differential equations, physics simulation, and image processing.\footnote{See~\cite{chambolle2010introduction} for a more complete treatment.} Given a smooth function $f$ from some domain $U\subset\R^n$ to $\R$, we define the TV of $f$ as
\[ |u|_{TV} = \int_U |\grad f|. \]
In the special case where $n=1,$ this is the total rise and fall of the function, hence the name. An important special case is when $n=2$ or $3$ and $f$ is the indicator function of a region $V\subset U$. In such a case, $|f|_{TV}$ is the perimeter or surface area of $V$.

Total variation minimization is an important heuristic in image processing, where e.g.\ a black and white image that is corrupted by noise can be viewed as a function $f: [0,1]^2 \to [0,1]$, where the value of $f$ varies from $0$ (black) to $1$ (white). A common task is to remove the noise and recover the original image. Since noise is manifest as large gradients in $f$, early approaches found $u$ as the solution to a minimization problem such as
\[
	\min_u \int_{[0,1]^2} ||\grad u||^2 + ||u-f||^2.
	\label{eqn:l2_regularized}
\]
The solution to such a problem is a smoothed image, which means that the noise is eliminated, but all edges are also erroneously eliminated. The correct approach~\cite{rof_1992}, is to modify the problem as follows
\[
	\min_u \int_{[0,1]^2} ||\grad u|| + ||u-f||^2.
\]
This small change allows the minimization procedure to preserve edges and yields much better results in many applications. The reason is that minimizers of total variation tend to be \emph{piecewise} smooth. Total variation minimization has other applications as well, such as compressed sensing~\cite{candes_romberg_tao_2006} and mean curvature flow~\cite{chambolle2009total,kolmogorov2007}.

Network community detection is in some ways analogous to image segmentation, in that both seek a partition into coherent subsets, and one of the main ideas behind the use of total variation in the network context is that it helps us arrive at the ``correct'' energy to optimize for, as in the image processing context. An important example is spectral approaches, such at those of~\cite{newman_spectral,newman_spectral_2}. In the case of only two communities, we let $u$ be a real-valued function on the nodes of the graph. A partition of the nodes into two communities can be encoded in such a function by letting $u=1$ on the nodes in one set and $u=-1$ on the others. The modularity can then be written as
\begin{align}
	\frac{1}{4m}	\sum_{ij} \left(w_{ij} - \gamma \frac{k_ik_j}{2m}\right)(1+u_iu_j)
	&= \frac{1}{4m}	\sum_{ij} \left(w_{ij} - \gamma \frac{k_ik_j}{2m}\right)
	+ \frac{1}{4m}	\sum_{ij} \left(w_{ij} - \gamma \frac{k_ik_j}{2m}\right)u_iu_j \\
	&= \const + \frac{1}{4m} \sum_{ij} \left(w_{ij} - \gamma \frac{k_ik_j}{2m}\right)u_iu_j \\
	&= \const + \frac{1}{2} u^T M u
	\label{exp:modularity_matrix}
\end{align}
where $M_{ij} = w_{ij} - \gamma \frac{k_i k_j}{2m}$ is the \emph{modularity matrix}.
Thus,~\cref{exp:modularity_matrix} is exactly equal to modularity when $u$ represents a partition but has an obvious extension to all $N$-vectors. An important idea in spectral approaches is to maximize~\cref{exp:modularity_matrix} or related energies over all real vectors and then employ some kind of thresholding on the values of the result to recover a binary partition. Recursive bipartitioning can be used to find partitions into more than two communities. A large number of variations on this idea exist and are widely used. Such approaches are analogous to ideas from~\cref{eqn:l2_regularized}, in that the solutions are expected to be smooth because of the quadratic term, which is indeed observed in practice, thus necessitating some kind of thresholding. In contrast, by using a non-quadratic measure of differences in the value of $u$ across edges, it is possible to promote sharp interfaces in the solutions.

We now very briefly give the definition of total variation on a graph, referring to~\cite{gilboa_osher_2008,van_gennip_2014} for a complete treatment, where it is shown that these choices are consistent with discrete notions of Riemannian metrics, inner products, divergences, and so forth. The nonlocal gradient of a function $f:G \to \R$ at node $i$ in the direction of the edge from $i$ to $j$ is
\[
	\nabla f(i,j) = f(j) - f(i).
\]
The \emph{graph total variation} is then given by the 1-norm of $\nabla f$ at node $i$
\begin{equation}
	|f|_{TV} = \frac{1}{2} \sum_{ij} w_{ij}|f(j) - f(i)|,
	\label{eqn:vector_tv}
\end{equation}
where $w_{ij}$ is the $i$, $j$ entry of the adjacency matrix (see~\cref{def:G}). We will actually use a slight generalization of~\cref{eqn:vector_tv} to the case where $f: \{1,\cdots,N\} \to \R^{\nhat}$ is vector-valued, in which case
\[
	|f|_{TV} = \sum_{\ell=1}^{\nhat} |f_\ell|_{TV}
\]
where $f_\ell$ is the $\ell$-th component of $f$. It is usually convenient in this case to identify $f$ with an $N\times\nhat$ matrix where $f_{i\ell} = f_{\ell}(i)$. Then we have
\[
	|f|_{TV} = \sum_{\ell=1}^{\nhat} \frac{1}{2} \sum_{ij=1}^N w_{ij}|f_{i\ell} - f_{j\ell}|.
	\label{eqn:matrix_tv}
\]

Graph total variation is connected to graph cuts, which correspond roughly to perimeter in Euclidean space.
\begin{definition}
	Let $S$ be a subset of the nodes of $G$. Then the graph cut associated to $S$ is given by
	\[
		\Cut(S,S^c) = \sum_{i\in S, j\in S^c} w_{ij}.
	\]
\end{definition}
Let $f:\{1,\cdots,N\}\to\R$ be the characteristic function of a set of nodes $S$. Then we can calculate
\begin{align}
	|f|_{TV} &= \frac{1}{2}\sum_{ij} w_{ij}|f(i) - f(j)|
	= \sum_{i\in S, j\in S^c} w_{ij}
	= \Cut(S,S^c).
\end{align}
TV minimization on a graph tends to produce piecewise-constant functions whose corresponding graph cut is small~\cite{egil_coarea}.


\section{Equivalence Theorem and its Consequences}
\label{sec:theorem}
In this section, we derive representations of modularity and explore some consequences. We will need definitions:
\begin{definition}
	A family of sets $S_1,\cdots,S_{\nhat}$ is a partition of a set $S$ if $S = \bigcup_{\ell=1}^{\nhat} S_{\ell}$ and $S_{\ell_1} \cap S_{\ell_2}$ is empty for each $\ell_1 \ne \ell_2$.
\end{definition}
\begin{definition}
	Let $\Pi(G)$ be the set of all partitions of the nodes of $G$.
	For each partition $A_1,\cdots,A_{\nhat}$ in $\Pi(G)$, there is an $N\times\nhat$ \emph{partition matrix} defined by,
	\[
		u_{i\ell} =
		\begin{cases}
			1 & i \in A_{\ell} \\
			0 & i \in A_{\ell}^c
		\end{cases}
	\]
	For a matrix $u$, we say $u\in \Pi(G)$ when $u$ is the partition matrix of some partition.
\end{definition}
\begin{definition}
	For any subset $S$ of the nodes of $G$, its \emph{volume} is given by $\vol S = \sum_{i\in S} k_i$.
\end{definition}

\subsection{Formulations of Modularity on Terms of TV and Graph Cuts}
We are now ready to give the different formulations of modularity that form the basis for our subsequent analysis.
\begin{proposition}[Equivalent forms of modularity]
	The following optimization problems all have the some solution set:
	\begin{align}
		&\text{Modularity:} &
		\label{eqn:modularity_opt}
		& \underset{\nhat\in\N, \{A_\ell\}_{\ell=1}^{\nhat} \in \Pi(G)}{\argmax}
		& & \sum_{\ell=1}^{\nhat} \sum_{ij\in A_{\ell}} w_{ij} - \gamma \frac{k_ik_j}{2m} \\
		&\text{Balanced cut (I):} &
		\label{eqn:balanced_cut}
		& \underset{\nhat\in\N, \{A_\ell\}_{\ell=1}^{\nhat} \in \Pi(G)}{\argmin}
		& & \sum_{\ell=1}^{\nhat} \left( \Cut\left( A_\ell,A_\ell^c \right) + \frac{\gamma}{2m} \left( \vol A_\ell \right)^2\right)\\
		&\text{Balanced cut (II):} &
		\label{eqn:balanced_cut_2}
		& \underset{\nhat\in\N, \{A_\ell\}_{\ell=1}^{\nhat} \in \Pi(G)}{\argmin}
		& & \sum_{\ell=1}^{\nhat} \left( \Cut\left( A_\ell,A_\ell^c \right) + \frac{\gamma}{2m} \left( \vol A_\ell - \frac{2m}{\nhat} \right)^2 \right) + \gamma \frac{2m}{\nhat} \\
		&\text{Balanced TV (I):} &
		& \quad\underset{\nhat\in\N, u\in \Pi(G)}{\argmin}
		\label{eqn:balanced_tv}
		& & |u|_{TV} + \frac{\gamma}{2m}\left|\left|k^Tu\right|\right|_2^2  \\
		&\text{Balanced TV (II):} &
		& \quad\underset{\nhat\in\N, u\in \Pi(G)}{\argmin}
		\label{eqn:balanced_tv_2}
		& & |u|_{TV} + \frac{\gamma}{2m}\left|\left|k^Tu - \frac{2m}{\nhat}\right|\right|_2^2 + \gamma \frac{2m}{\nhat}
	\end{align}
	\label{main_theorem}
\end{proposition}
Each of the preceding forms has a different interpretation. The original formulation of modularity was based on comparison with a statistical model and views communities as regions that are more connected than they would be if edges were totally random. The cut formulations represent modularity as favoring sparsely interconnected regions with balanced volumes, and the TV formulation seeks a piecewise-constant partition function $u$ whose discontinuities have small perimeter, together with a balance-inducing quadratic penalty. The cut and TV forms come in pairs. The first form (labelled ``I'') is simpler to write but harder to interpret, while the second (labelled ``II'') has more terms, but the nature of the balance term is easier to understand, as it is minimized (for fixed $\nhat$) when each community has volume $2m/\nhat$. Furthermore, the third term of the forms labelled II reveals that the incentive to increase the number of communities $\nhat$ can be quantified in terms of an $O(\nhat^{-1})$ penalty term, which is not obvious from other formulations of modularity.

One can compare these equivalent formulations with~\cite{hui}, in which minimizing the functional
\begin{equation}
	|u|_{TV} - \gamma || u - \mean(u) ||_{\ell^2(G)}^2 = |u|_{TV} - \gamma \sum_{i\ell} k_i \left|u_{i\ell} - \frac{1}{2m} \sum_{i' = 1}^N k_i u_{i'\ell}\right|^2
	\label{eqn:hui_formulation}
\end{equation}
is shown to be equivalent to modularity optimization, subject to the same constraint as the other TV formulas presented here. Thus, in~\cite{hui}, there are two sources of nonconvexity, namely the balance term and the constraint, while in our formulation, the discrete constraint is the only source of nonconvexity.%
\footnote{To see rigorously that~\cref{eqn:hui_formulation} is nonconvex, consider the special case of two nodes connected by a single edge, $\gamma=1$ and $u=[\lambda\ 0;0\ 0].$ Then considering~\cref{eqn:hui_formulation} as a function of $\lambda$ immediately shows the nonconvexity. The nonconvexity is actually very general; computing the second derivative of the second term in~\cref{eqn:hui_formulation} with respect to any component of $u$ gives a negative value for any connected graph with more than one node. Since the TV term grows asymptotically linearly, it is eventually dominated by the quadratic growth of the second, concave term.}
It is also clearer from our formulation which features of a solution are incentivized by modularity optimization, namely, the two priorities of having a small graph cut and balanced class sizes are the only considerations. The relative weight of these considerations, as well as the number of communities, is governed by $\gamma$, via the second and third terms of~\cref{eqn:balanced_tv_2}. Overall, these theoretical simplifications make the nonconvexity of the problem easier to navigate.

We note that forms similar to~\crefrange{eqn:balanced_cut}{eqn:balanced_tv_2} have appeared in the literature before (see e.g.~\cite{reichardt_bornholdt_2006}), although the only previous work to consider any modularity formula in terms of total variation is~\cite{hui}. To the best of our knowledge, the composition of modularity into the three intuitively meaningful terms in the forms labelled II is also novel. We will see shortly that the total variation perspective on~\crefrange{eqn:balanced_cut}{eqn:balanced_tv_2}, combined with the convexity of the functionals in~\cref{eqn:balanced_tv} and \cref{eqn:balanced_tv_2} leads to a number of new developments.

%


\Crefrange{eqn:balanced_cut}{eqn:balanced_tv_2} provide a convenient way to incorporate metadata into the partitioning process. This can be done by simply incorporating a fidelity term and minimizing the functional
\begin{equation}
	|u|_{TV} + \frac{\gamma}{2m}||k^Tu||_2^2 + \lambda ||\chi*(u - f) ||_2^2
	\label{eqn:ssl}
\end{equation}
where $\lambda>0$ is a parameter, $f$ is a term containing the metadata labels, $*$ is the entry-wise matrix product, and $\chi$ is a matrix that is zero except in the entries where labels are known. Including metadata should always be done with care, of course, but the general utility of semisupervised learning is well-attested in image processing and machine learning applications. (See~\cref{ssl_table} for two numerical examples.)

\begin{proof}[Proof of \Cref{main_theorem}]
	Notice that the cut and TV formulations are really just a change of notation, so that there are two nontrivial equivalences, namely the equivalence of~\cref{eqn:modularity_opt} with~\cref{eqn:balanced_cut} and the equivalence of~\cref{eqn:balanced_cut} and~\cref{eqn:balanced_cut_2}. We first show the equivalence of~\cref{eqn:modularity_opt} with~\cref{eqn:balanced_cut}.
	Fix $\nhat,$ and consider an otherwise arbitrary partition $\{A_1,\dots,A_{\nhat}\}$ of $G$. Then we have
	\begin{align}
		Q &= \frac{1}{2m} \sum_{\ell =1}^{\nhat} \sum_{ij \in A_\ell } w_{ij} - \gamma\frac{k_ik_j}{2m} \\
		&= \frac{1}{2m} \sum_{\ell =1}^{\nhat}\left(\sum_{i\in A_\ell, j\in\{1,\ldots,N\}} w_{ij} - \sum_{i \in A_\ell, j \in A_\ell^c } w_{ij}\right) - \frac{\gamma}{2m} \sum_{\ell =1}^{\nhat} \sum_{ij \in A_\ell }\frac{k_ik_j}{2m} \\
		&= \frac{1}{2m} \sum_{ij=1}^N w_{ij} - \frac{1}{2m}\sum_{\ell =1}^{\nhat} \sum_{i \in A_\ell, j \in A_\ell^c } w_{ij} - \frac{\gamma}{2m} \sum_{\ell =1}^{\nhat} \sum_{ij \in A_\ell }\frac{k_ik_j}{2m} \\
		&= 1 - \frac{1}{2m} \sum_{\ell =1}^{\nhat}\sum_{i \in A_\ell, j \in A_\ell^c } w_{ij} - \frac{\gamma}{2m} \sum_{\ell =1}^{\nhat} \sum_{ij \in A_\ell }\frac{k_ik_j}{2m} \\
		&= 1 - \frac{1}{2m}\sum_{\ell =1}^{\nhat} \Cut(A_\ell ,A_\ell^c) - \frac{\gamma}{2m}\sum_{\ell=1}^{\nhat}\sum_{ij \in A_\ell} \frac{k_ik_j}{2m}. \\
		\intertext{Summing along the $j$ index first yields}
		&= 1 - \frac{1}{2m}\sum_{\ell=1}^{\nhat}\left( \Cut(A_\ell,A_\ell^c) + \frac{\gamma}{2m} \sum_{\ell=1}^{\nhat}\sum_{i \in A_\ell} k_i\vol A_\ell \right)\\
		&= 1 - \frac{1}{2m}\sum_{\ell=1}^{\nhat}\left( \Cut(A_\ell,A_\ell^c) + \frac{\gamma}{2m}(\vol A_\ell)^2\right)
	\end{align}
	Thus, the maxima of modularity coincide with the minima the functional from~\cref{eqn:balanced_cut}, as required.

	To see that~\cref{eqn:balanced_cut} and~\cref{eqn:balanced_cut_2} are equivalent, we calculate:
	\begin{align}
		& \sum_{\ell=1}^{\nhat} \left( \Cut\left( A_\ell,A_\ell^c \right) + \frac{\gamma}{2m} \left( \vol A_\ell - \frac{2m}{\nhat} \right)^2 \right) \\
		&= \sum_{\ell=1}^{\nhat} \left( \Cut\left( A_\ell,A_\ell^c \right)
		+ \frac{\gamma}{2m} \left(
		\left( \vol A_\ell \right)^2 - \frac{4m}{\nhat}\vol A_\ell  + \frac{4m^2}{\nhat^2}
		\right) \right)\\
		&= \sum_{\ell=1}^{\nhat}
		\left(
		\Cut\left( A_\ell,A_\ell^c \right)
		+ \frac{\gamma}{2m} \left(\vol A_\ell \right)^2
		\right)
		- \frac{\gamma}{2m}\frac{8m^2}{\nhat} + \frac{\gamma}{2m}\frac{4m^2}{\nhat}\\
		&= \sum_{\ell=1}^{\nhat}
		\left(
		\Cut\left( A_\ell,A_\ell^c \right)
		+ \frac{\gamma}{2m} \left(\vol A_\ell \right)^2
		\right)
		- \gamma \frac{2m}{\nhat}
	\end{align}
\end{proof}

\subsection{On convex relaxations}
The preceding equivalence theorem makes it very tempting to look for a convex relaxation of~\cref{eqn:balanced_tv}. Recall that, given two sets, $A\subset B$ where $A$ is discrete and a functional $\mathcal{F}:A\to \R$, a relaxation of $\mathcal{F}$ is any function $\bar{\mathcal{F}}:B\to \R$ such that $\mathcal{F} = \bar{\mathcal{F}}$ on $A$. A relaxation is called exact in the context of minimization if $\min_{x\in A}\mathcal{F} =\min_{x\in B} \bar{\mathcal{F}}$.\footnote{Analogous notions apply to maximization problems, but we are using~\cref{eqn:balanced_tv} rather than~\cref{eqn:modularity_opt} for the moment.} Finally, a relaxation is called convex if $\bar{\mathcal{F}}$ is convex.

Modularity~\cref{eqn:modularity_opt} and balanced TV~\cref{eqn:balanced_tv} are both defined only over a discrete domain, and we would like an extension, or relaxation, of these functions to a larger, continuum domain so that they are easier to work with numerically. Ideally, we could arrive at a convex relaxation and have access to the powerful tools of convex optimization. The formulation in~\cref{eqn:balanced_tv} indicates one way to proceed. Using~\cref{eqn:balanced_tv}, we already have a convex functional except for the domain, so one would hope that the obvious relaxation obtained by using formula~\cref{eqn:balanced_tv} on all of $\R^{N\times\nhat}$ would be useful. Unfortunately, the next theorem shows that this obvious relaxation is minimized by the constant matrix and is thus not likely to be useful. In fact, it shows that a large class of other convex relaxations will be uninformative. This will force us to look for nonconvex approaches in the next subsection. Before we state the theorem, we include three more definitions:
\begin{definition}
	The symmetric group on $\nhat$ symbols, $S_{\nhat}$, is the set of all permutations on $\{1,\cdots,\nhat\}$. Each element $\sigma\in S_{\nhat}$ acts on a matrix $u\in\R^{N\times\nhat}$ with columns $u_1,\cdots,u_{\nhat}$ by sending $u$ to another matrix, $\sigma(u)$ with columns $u_{\sigma(1)},\cdots,u_{\sigma(\nhat)}.$
	If $u\in\Pi(G)$, then $\sigma(u)$ is the same partition with the labels permuted.
\end{definition}
\begin{definition}
	A map $\F$ from some set of matrices to the real numbers is \emph{symmetric} if it is invariant under column permutations,
	i.e. $\F(u) = \F(\sigma(u))$ for all $\sigma$ and $u$.
	\label{def:sym}
\end{definition}
The balanced TV functional~\cref{eqn:balanced_tv} is symmetric, and most natural relaxations of it are symmetric.
\begin{definition}
	Given a set $S$ lying in a vector space $V$, the convex hull is the smallest convex set containing $S$.
\end{definition}
It can be shown that in a finite-dimensional vector space, the convex hull exists and is the intersection of all convex sets containing $S$. For example, if $S$ is given by three noncolinear points in the plane, the convex hull is a triangle.

We now state and prove our theorem on convex relaxations of modularity.
\begin{theorem}
	Let $\mathcal{F}$ be given by~\cref{eqn:balanced_tv} with domain $\Pi(G,\nhat) = \Pi(G)\cap\R^{N\times\nhat}$, and let $\tilde{\mathcal{F}}$ be any symmetric, convex extension of $\mathcal{F}$ to the convex hull of $\Pi(G,\nhat)$. Then $\tilde{\mathcal{F}}$ has a trivial, global minimizer $\tilde{u}$ that has all columns equal to each other, thus yielding no classification information.

	If the symmetry requirement is dropped, then $\tilde{u}$ need not be a global minimizer, but will have an objective value at least as good as any $u\in\Pi(G,\nhat)$.
\end{theorem}
\begin{proof}
	We consider the symmetric case first.
	Let $u$ lie in the convex hull of $\Pi(G,\nhat)$. We will use the symmetry of $\tilde{\mathcal{F}}$ plus convexity to average all the column permutations of $u$ and get a value of $\tilde{\mathcal{F}}$ at least as low as $u$ gives. Let $\tilde{u} = \frac{1}{\nhat !} \sum_{\sigma\in S_{\nhat}} \sigma(u)$. Then by Jensen's inequality we have
	\[\tilde{\mathcal{F}}(\tilde{u}) = \tilde{\F}\left(\frac{1}{\nhat !}\sum_{\sigma\in S_{\nhat}} \sigma(u)\right) \le \frac{1}{\nhat !} \sum_{\sigma\in S_{\nhat}} \tilde{\F}(\sigma(u)) = \tilde{\F}(u).\]
	Since $u$ was arbitrary, $\tilde{u}$ is a global minimizer.

	Finally, all the columns of $\tilde{u}$ are equal,\footnote{Incidentally, all of the rows are also equal, since row stochasticity is preserved under column permutation.} and thus uninformative. To see this, take any $k,\ell\in\{1,\ldots,\nhat\}$. Let $\tau$ be the permutation that swaps these two values and leaves all the others fixed. Then any $\sigma\in S_{\nhat}$ can be written uniquely as $\tau\circ\sigma'$, with $\sigma' = \tau\circ\sigma$. (Proof: $\tau\circ\tau$ is the identity, so left-multiply by $\tau$.) Thus the $k$-th column of $\tilde{u}$ is given by
	\begin{align}
		\tilde{u}_k &= \frac{1}{\nhat !}\sum_{\sigma\in S_{\nhat}} \sigma(u)_k  \\
		&= \frac{1}{\nhat !}\sum_{\sigma'\in S_{\nhat}}  \tau\circ\sigma'(u)_k  \\
		&= \frac{1}{\nhat !}\sum_{\sigma'\in S_{\nhat}}  \sigma'(u)_{\ell}  \text{    \hspace{.2in} (Note the change in subscript!)}\\
		&= \tilde{u}_{\ell}
	\end{align}
	So all columns of $\tilde{u}$ are equal.

	The non-symmetric case is similar, except that $u$ must lie in $\Pi(G,\nhat)$ since $\tilde{\mathcal{F}}$ is not known to be symmetric. Therefore, in that case, we can only show that the value of $\tilde{\F}$ at $\tilde{u}$ is at least as good as at any point in $\Pi(G,\nhat)$.
\end{proof}
This means that modularity cannot be convexly relaxed using this embedding of $\Pi(G,\nhat)$ in $\mathbb{R}^{N\times\nhat}$.\footnotemark  Thus, our only option to make use of smooth optimization techniques is a non-convex relaxation. In the following subsection, we present one such family of relaxations.
\footnotetext{We do note, however, that by means of a different embedding~\cite{chen_modularity} was able to obtain a convex relaxation with solutions which, while not discrete, are also not trivial. Thus, the embedding requirement is a non-trivial part of our theorem. Other related works include~\cite{chan_modularity} and~\cite{agarwal}.

Note that our proof does not rely on many specific properties of modularity, and indeed, a similar theorem holds for any symmetric quality function over a discrete domain.}

\subsection{Ginzburg-Landau Relaxation}

In this subsection, we develop a way to relax the modularity problem to a continuum domain, which can make the nonconvexity more manageable. In other TV problems arising in materials science and image processing, discrete constraints similar to modularity's are dealt with using the idea of \emph{phase fields}, where a thin transition layer between discrete-valued regions is allowed, making the problem smooth so that it can be attacked by continuum methods. (See e.g.~\cite{gl_example,esedoglu2006threshold,ambrosio_tortorelli,bertozzi_flenner_2012}.) As discussed above, TV is used for two of its properties: promoting small perimeter and encouraging binary results. The Ginzburg-Landau relaxation replaces the TV term with two other terms: the Dirichlet energy and a multiwell potential, each of which has one of the aforementioned properties. Thus the Ginzburg-Landau energy in the continuum is given by
\[
	F_{\epsilon}(u) = \int_U \epsilon ||\nabla u(x)||^2 + \frac{1}{\epsilon}P(u(x))\, dx,
\]
where $\epsilon$ is a small parameter and $P$ is a multiwell potential with local minima at the corners of the simplex, which is the set of nonnegative vectors whose components sum to $1$. The exact form of $P$ will not be important for our purposes, but we will give a concrete example in the next theorem.
A classical result asserts that for $u : U\subset\R \to \R$ and $P$ having minima at $0$ and $1$, we have the following convergence\footnotemark result:
\[
	F_{\epsilon}(u) \xrightarrow{\Gamma}
	\left\{\begin{array}{lr}
		\const |u|_{TV} & \text{if $u$ is binary} \\
		+\infty   & \text{otherwise}
	\end{array} \right.
\]
\footnotetext{See the appendices for an overview of $\Gamma$-convergence.} as $\epsilon\to 0$, under appropriate conditions.

In order to arrive at the graph Ginzburg-Landau functional, observe that if we ignore boundary terms, then integration by parts gives
\begin{align}
	\int_U ||\nabla u||^2 &= \int_U \nabla u\cdot\nabla u
	= \int_U - \Div \nabla u\cdot u
	= \int_U - \Delta u\cdot u,
\end{align}
which suggests that we use a graph Laplacian in our formulation. The Laplacian that is appropriate for our context is the \emph{combinatorial} or \emph{unnormalized Laplacian}, $L = \diag(k) - W.$

In~\cite{bertozzi_flenner_2012}, the idea of using a Ginzburg-Landau functional in graph-based optimization first appeared, and it has subsequently been treated in more depth in~\cite{van_gennip_2012}, where much of the continuum theory was successfully extended to graphs. Our approach closely mirrors~\cite{hui}, the main difference in this case simply being that our functionals have better convexity properties, which allows for different estimates and improved techniques. We begin with a convergence result.
\begin{theorem}[$\Gamma$-convergence for the balanced TV problem]
	Assume $\frac{P(u_i)}{||u_i||}\to\infty$ as $||u_i||\to\infty$, where $u_i$ is the $i$-th row of $u$.
	Then the functionals\footnote{Note that due to the discrete setting, there is no epsilon factor preceding the Laplacian term, see~\cite{van_gennip_2012}.}
	\begin{align}
		\mathcal{F}_{\epsilon} &= ||\nabla u||^2_2 + \frac{1}{\epsilon} \sum_{i=1}^N P(u_i) + \frac{\gamma}{2m}||k^Tu||_2^2 \\
		&:= u^TLu + \frac{1}{\epsilon} \sum_{i=1}^N P(u_i) + \frac{\gamma}{2m}||k^Tu||_2^2,
		\label{eqn:GL}
	\end{align}
	defined over all of $R^N$,
	$\Gamma$-converge to the functional
	\begin{align}
		\begin{cases}
			|u|_{TV} + \frac{\gamma}{2m} || k^Tu||_2^2	&	\text{if $u$ corresponds to a partition} \\
			+\infty 					& 	\text{otherwise}
		\end{cases}
	\end{align}
	In particular,
	\begin{itemize}
		\item for any sequence $\epsilon_n \to 0$, and any corresponding sequence $u_{\epsilon}$ of minimizers of $\mathcal{F}_{\epsilon_n}$, there is a subsequence that converges to a maximizer of modularity, and
		\item any convergent subsequence of the $u_{\epsilon}$ converges to a maximizer of modularity.
	\end{itemize}
	\label{thm:GL}
\end{theorem}
The proof is given in the appendices.

Moving forward, we focus on minimizing the relaxed functionals from~\cref{thm:GL}. While using the Ginzburg-Landau functional does introduce a Laplacian into our formulation, we stress that this approach is different from spectral approaches, such as those in~\cite{newman_spectral,newman_spectral_2}---the preceding result on $\Gamma$-convergence shows that the real object we are aiming for is TV, which, as discussed in the background section, has very different solutions from quadratic optimization problems. In the results section, we will see numerically that the answers are indeed different from one particular spectral method.

\section{Numerical Scheme}
\label{sec:numerics}

\subsection{MBO iteration}

We minimize the functional from~\cref{eqn:GL} using an adaptation of the graph MBO scheme. We call our approach Balanced TV. The acronym ``MBO'' stands for Merriman, Bence, and Osher~\cite{MBO_1992}, who introduced this algorithm in Euclidean space. It has been widely used as an approach to motion by mean curvature and TV minimization. The connection between graph-based TV and MBO was first made in~\cite{merkurjev_kostic_bertozzi_2013} and~\cite{garcia_merkurjev_bertozzi_percus_2014}. The theoretical study of the algorithm on graphs was initiated in~\cite{van_gennip_2014}. We sketch the logic of MBO here and refer the reader to \cite{MBO_1992} for a more complete treatment. The scheme works by approximating the gradient descent flow of the Ginzburg-Landau functional in the case where $\epsilon$ is very small. Consider the Ginzburg-Landau gradient descent equation (at fixed $\nhat$)
\[ \frac{d}{dt} u = -Lu - \frac{1}{\epsilon} P'(u) -\frac{\gamma}{m} k k^T u.\]
One way to approximate this flow is by operator splitting~\cite[p.22]{tai_splitting} with time-step $dt$ and $t_n =n*dt, n = 0, 1,2, \cdots$. Given $u^n$ one obtains $u^{n+\frac{1}{2}}$ as the solution to
\begin{equation}
	\begin{split}
		\frac{d}{dt} u_1  = -Lu_1  -\frac{\gamma}{m} k k^T u_1, \quad t \in [t_n, t_{n+1}], \\
		u_1 (t_n) = u^n, u^{n+1/2} = u_1(t_{n+1}).
	\end{split}
\end{equation}
Then one gets $u^{n+1}$ by solving
\begin{equation}
	\begin{split}
		\frac{d}{dt} u_2  = - \frac{1}{\epsilon} P'(u_2), \quad t \in [t_n, t_{n+1}],\\
		u_2 (t_n) = u^{n+1/2}, u^{n+1} = u_2(t_{n+1}).
	\end{split}
\end{equation}
The iteration continues until a fixed point is reached.
Such operator splitting schemes are typically first-order accurate in time. In the case where $\epsilon$ is very small, the second flow is essentially a thresholding operation, pushing all values of $u$ into the nearest well, i.e.
\[ u^{n+1}_{i\ell} = \left\{
	\begin{array}{rl}
		1 & \ell = \argmax_{\hat{\ell}} u^{n + \frac{1}{2}}_{i\hat{\ell}} \\
		0 & \text{otherwise}
	\end{array}
	\right.
\]
This gives the MBO scheme:
\begin{tcolorbox}
	\fbox{Balanced TV MBO scheme}\\
	\begin{algorithmic}
		\State Initialize $u$ randomly.
		\State Set $n=0$.
		\While{A stationary point has not been reached}
		\State $u^{n+\frac{1}{2}} = e^{-dt M}u^n$ where $M = L + \frac{\gamma}{m}kk^T$
		\State $u^{n+1} = \threshold(u^{n+\frac{1}{2}}$)
		\State $n = n+1$
		\EndWhile
	\end{algorithmic}
\end{tcolorbox}
The most expensive part of this procedure is evaluating the matrix exponential. We accomplish this efficiently using a pseudospectral scheme, which will be described below.

We treat the forcing term implicitly, which differs from several recent studies, such as~\cite{hui,bertozzi_flenner_2012,merkurjev_kostic_bertozzi_2013}. This can be done efficiently because the operator $M$ is positive semi-definite and can be applied to a vector in linear time, assuming $A$ is sparse. Implicit treatment has the advantage of avoiding an inner loop, which is time-consuming, has a timestep-restriction, and adds another user-set parameter, namely the inner loop timestep. For this reason, the implicit treatment described herein is much easier and faster than the typical nested-loop approach.

As stated, we assume from here on that $A$ is sparse. The case where $A$ is dense could be approached using the Nystr\"om method, as in~\cite{bertozzi_flenner_2012}. Beware, however, that one must find a way to estimate $k$ and $2m$ efficiently, which is not obvious. An alternative is to sparsify the network in preprocessing, which is the approach taken in our examples. This is generally cheap compared to the cost of partitioning the resulting sparse network.
\subsection{Treating the matrix exponential}

As stated above, the most time-intensive step in the MBO iteration is the matrix exponential, and this step is repeated many times. Therefore, it makes sense to use a pseudospectral scheme, as described in, for instance, \cite{bertozzi_flenner_2012}. This means that we precompute the eigenvalues and eigenvectors of $M$, and use them to solve the matrix exponential. By doing the eigenvalue calculation up front, each iteration is greatly accelerated. Here is how the scheme looks:
\begin{tcolorbox}
	\fbox{Pseudospectral Balanced TV MBO scheme}\\
	\begin{algorithmic}
		\State Initialize $u$ randomly.
		\State Calculate the eigenvalues of $M$, and form the diagonal matrix $D$ with its diagonals being the eigenvalues.
		\State Also calculate the eigenvectors and form the matrix $V$ whose columns are the eigenvectors.
		\While{a stationary point has not been reached}
		\State $a^n = V^T u^{n}$.
		\State $a^{n+1} = e^{-dt D}a^n$
		\State $u^{n+\frac{1}{2}} = V a^{n+1}$
		\State $u^{n+1} = \threshold(u^{n+\frac{1}{2}})$.
		\EndWhile
	\end{algorithmic}
\end{tcolorbox}
In practice, it may not be possible to calculate the full spectrum of $M$, if $M$ is large. In this case, we calculate the $N_\textrm{eig}$ smallest eigenvalues and eigenvectors of $M$. Then instead of changing coordinates using a full matrix, use the $N\times N_\textrm{eig}$ matrix $V$ exactly the same way as before. This is equivalent to projecting onto a subspace generated by these eigenvectors, and it makes the algorithms very efficient.

To understand the effect of computing only a few eigenvectors, recall that $M$ is positive semi-definite. Therefore, it has an orthonormal basis of eigenvectors, and the evolution we are solving, namely $\frac{d}{dt} u = -Mu$, can be diagonalized as $a_t = -Da$ where $a = V^Tu$, and $V$ is the full matrix of eigenvalues, and $D$ is a non-negative, diagonal matrix. Therefore, the evolution occurs in distinct ``modes'', with rates of decay controlled by the eigenvalues of $M$. The modes corresponding to small eigenvalues persist longer than those corresponding to large eigenvalues (which experience stiff exponential decay), so that it is not a bad approximation to simply project these components away when it is numerically necessary. Thus, in practice, we collect the smallest eigenvectors of $M$ and the corresponding eigenvectors, neglecting the others.

We use Anderson's iterative Rayleigh-Chebyshev code~\cite{anderson_2010}---which the author kindly provided to us---to get the eigenvalues and eigenvectors.
We generally set $N_\textrm{eig} = 5\nhat$.

\subsection{Determining the number of communities}

The preceding algorithm assumes a fixed $\nhat$. In practice, we found three methods of determining the value of $\nhat$:
\begin{enumerate}
	\item Use domain knowledge---for instance, in two moons, it is known that there are two communities,
	\item Try several values of $\nhat$ and take whichever one produces the best modularity---this works best in cases where there are few communities, as in MNIST. Note that the most time consuming part of the MBO scheme, namely computation of eigenvectors need only be done once, so that several different values of $\nhat$ can be tried without incurring much extra cost.
	\item Recursively partition the network---this works when many communities are present, as in the LFR networks.
		The partition is only made at each step if it increases modularity. This approach worked well in our examples, although in the case of LFR, where $O(N)$ communities are present, a lot of recursion is needed. This is compensated for by the fact that the subgraphs grow smaller and smaller near the end.
\end{enumerate}
%
\subsection{Scaling}

We expect the scaling of our approach to be roughly linear, as suggested by the following informal argument. The main components of the algorithm are
\begin{enumerate}
	\item
		finding eigenvalues and eigenvectors (probably $O(N\log^q N)$ for some $q$),\footnotemark
	\item
		changing coordinates using only the leading eigenvectors ($O(N)$ per iteration, with empirically $O(1)$ iterations needed to converge),
	\item
		evaluating the exponential of a vector componentwise (also $O(N)$ per iteration), and
	\item
		thresholding ($O(N)$ per iteration).
\end{enumerate}
\footnotetext{There is no rigorous result for the Rayleigh-Chebyshev procedure, but numerical evidence suggests strongly better than quadratic convergence, and $O(N\log^qN)$ is the convergence speed for some similar algorithms.}
The preceding estimates all apply in the case where no recursion is needed, i.e.\ the number of communities is known in advance. If the recursion is done by partitioning the graph into $\nhat$ pieces at each level, then the cost is heuristically on the order of
\[ \tilde{O}(N) + \nhat \tilde{O}\left(\frac{N}{\nhat}\right) + \nhat^2 \tilde{O}\left(\frac{N}{\nhat^2}\right) + \cdots + O(N) O(1) = \tilde{O}(N) \]
where $\tilde{O}$ means that logarithmic terms are neglected, and each term in the sum is the product of the number of partitioning problems to be solved with the size of the partitioning problems. This scalability is roughly borne out in our example data sets, although we warn that there are additional complications, based on the varying number of communities to be produced, differences in the efficiency of parallelization at different scales, and possibly other factors.

\subsection{On the choice of timestep}
Our approach requires the selection of parameters $\gamma,dt,$ $N_\textrm{eig}$, $\nhat$, and various other parameters and methods. In order to simplify the exploration of this parameter space in practical applications, it is useful to have some theory about the choice of these parameters. Here, we describe how to set $dt$ in the MBO scheme. This is especially useful in the recursive implementation, as the appropriate timestep empirically decreases as the graph gets smaller, and it would be laborious for a human to check at each recursion step. 

Our derivations are inspired by those in~\cite{van_gennip_2014}, and proofs are deferred to an appendix. First, we consider a lower bound on the timestep:
\begin{proposition}[Lower bounds on the timestep]
	Let $u_0\in\Pi(G,\nhat)$. If $u$ satisfies $\frac{d}{dt} u = -M u$ with initial data $u_0$, then we have the following bounds:
	\begin{enumerate}
		\item 
			
			\[
				||u(\tau) - u_0||_\infty \le e^{2 (\gamma + 1) k_{\mathrm{max}} \tau }.
			\]
		\item
			In the case where $\nhat = 2$, this bound implies that if the MBO timestep $\tau$ satisfies
			\[
				\tau < \frac{ \log 2 }{ 2(\gamma + 1) k_{\mathrm{max}} } \approx \frac{0.15}{(\gamma+1) k_{\mathrm{max}}},
			\]
			then the MBO iteration is stationary.
		\item 
			If $\rho$ is the spectral radius of $M$, we also have
			\[
				||u(\tau) - u_0||_\infty \le \sqrt{\nhat}||u_0||_2\left( e^{\tau\rho} - 1 \right).
			\]
		\item
			If $\nhat = 2$, the MBO iteration is guaranteed to be stationary whenever
			\[
				\tau < \rho^{-1} \log\left( 1 + N^{-\frac{1}{2}} \right).
			\]
	\end{enumerate}
	\label{prop:freezing}
\end{proposition}
Although we had to restrict to $\nhat=2$ in the above, we used the timestep restriction regardless of $\nhat$---indeed the authors expect that $\nhat=2$ is the worst case, although we are unable to prove it at present.

The upper bound on the timestep is more delicate. Normally, the upper bound would be determined by convergence theory, using error bounds and stability estimates, the theory of which is incomplete in the graph setting at present. Instead, we use the following heuristic to motivate our bounds: In most cases, $M$ is strictly positive definite, so the evolution $\frac{d}{dt} u = -Mu$ forces $u$ to decay toward $0$. The idea behind MBO is that the diffusion effects give information about curvature on short time scales, and the long time scales give information about more global quantities, which is useless in that context. Therefore, in the graph context, it makes sense to try to understand the time scale that is ``long'' and set the timestep to be shorter than that. Using the approach to $0$ as a convenient notion of long-time behavior, we obtain the following useful bounds:
\begin{proposition}[Decay estimates for $M$]
	Let $\frac{d}{dt} u = -M u$ with initial data $u_0\in\Pi(G,\nhat)$. Then the following bounds hold:
	\begin{enumerate}
		\item
			Assume $\lambda_1$ is the smallest eigenvalue of $M$. Then
			\[
				||u||_2 \le e^{-\tau \lambda_1} || u_0 ||_2.
			\]
		\item
			Let $M$ be nonsingular. Then for any $\epsilon > 0$, we have $||u(\tau)||_\infty < \epsilon$ if 
			\[
				\tau > \lambda_1^{-1} \log\left( \frac{||u_0||_2}{\epsilon} \right).
			\]
	\end{enumerate}
	\label{prop:steady_state}
\end{proposition}
In practice, setting the timestep as the geometric mean between this upper bound and the lower bound from Proposition~\ref{prop:freezing} has produced good results without resorting to hand-tuning of parameters.\footnote{We also found empirically that a simple time stepping procedure improved results sometimes: Let the algorithm run to convergence, then continue with a smaller timestep until convergence occurs again.}

\section{Results}
\label{sec:results}

\subsection{Summary}

Tables~\ref{six_table} and~\ref{hsi_table} summarize the results of our Balanced TV algorithm on several examples, mostly drawn from machine learning and image processing problems.~\footnote{We also performed some brief tests of our method on biological and social networks but found that the results were not as encouraging, apparently due to some structural differences from our machine learning networks---it would be interesting to understand this issue more.} We compared our method to the Modularity MBO algorithm from Hu et al.~\cite{hui}, as well as three other well-known algorithms: the Louvain method~\cite{blondel_2008}, the hierarchical method of Clauset, Newman, and Moore~\cite{clauset_modularity}, and a classic spectral recursive bipartitioning method of Newman~\cite{newman_spectral}. Our own method and that of Hu et al.\ were written in MATLAB except for the eigenvector computations, which use Anderson's Rayleigh-Chebyshev code~\cite{anderson_2010}, written in C++ with OpenMP support. The three other methods are slight modifications of igraph's C library implementations~\cite{igraph}. In practice, the difference in programming language may make a difference in speed, although the eigenvalue computation is typically the most time-intensive part of the computation. We chose a single conservative timestep for Modularity MBO rather than hand-tuning for each experiment. Our method and that of Hu et al.\ use a random starting seed, so we ran those codes 20 times and report the best modularity and classification rate and the median time. 

Overall, we found that our method is competitive with the state of the art on these data sets. Our method generally found higher-modularity partitions and had faster run times than either the method of Hu et al.\ or of Newman.\footnote{We chose this particular spectral method because it was available in igraph. A complete comparison with other spectral methods would be interesting but is beyond the scope of this paper.} The Louvain method and our method often gave similar modularity scores, although the partitions they uncovered were not necessarily similar. For example, on the MNIST example, our method achieved the better modularity score, but the Louvain partition matched the true labels more closely. On the Plume40 example, the opposite effect occurs, with our method achieving the lower modularity score but finding a partition that is closer to the true labeling of the pixels. Such issues are a manifestation of the well-known degeneracy of the modularity energy~\cite{clauset_2010}, where a number of dissimilar partitions can receive similarly high modularity scores. It is also an indication that modularity needs to be complemented with supervision, regularization, biased initialization, or some other device in order to reliably find the partition that is most appropriate for the problem. In Figure~\ref{ssl_table}, we illustrate the effectiveness of including a small amount of supervision with our method. (See~\cref{eqn:ssl}.)

\begin{table}
\setlength\tabcolsep{5.5pt} 
	\begin{tabular}{l|lllllll}
		\toprule
		&			&	{Moons}	&	{MNIST}	&      {LFR50k}	&	{Urban}	&	{Plume7}	&      {Plume40}\\
		\midrule                                                                                 
		&	Nodes		&	2,000	&	70,000	&	50,000	&	94,249	&	286,720	&$1.6* 10^6$	\\
		&	Edges		&$1.8*10^4$	&$4.7* 10^5$	&$7.9* 10^5$	&$6.8* 10^5$	&$5.3*10^6$	&$2.9* 10^7$	\\
		&	Communities 	& 	2	&	10	&	2,000	&	5	&	5	&	5	\\
		& 	Res. Param.	&	0.2	&	0.5	&	15	&	0.1	&	1	&	1	\\
		\midrule                                                                                
		\multirow{6}{*}{Modularity}                                                              
		&	Our method	& 	0.84	&	0.92	&	0.77	&	0.95	&	0.76	&	0.64	\\
		&	Hu et al.	& 	0.85	&	0.91	&	0.58	&	0.95	&	0.74	&	0.64	\\
		&	Hierarchical	&	0.77	& 	0.88	&	0.88	&	0.94	&	0.65	&	0.92	\\
		&	Louvain		&	0.72	& 	0.83	&	0.89	&	0.90	&	0.78	&	0.97	\\
		&	Spectral		&	0.60	& 	0.56	&	-5.88	&	0.90	&	0.30	&	0.04	\\
		&	Ref		&	0.83	& 	0.92	&	0.89	&	0.90	&	0.00	&	0.00	\\
		\midrule                                                                                
		\multirow{2}{*}{Classification}                                                         
		&	Our method	&	0.97	& 	0.90	& 	0.92	& 	----	&	----	&	----	\\
		&	Hu et al.	&	0.95	& 	0.80	& 	0.72	& 	----	&	----	&	----	\\
		&	Hierarchical	& 	0.98	&	0.93	& 	0.80	& 	----	& 	----	&	----	\\
		&	Louvain		& 	0.98	&	0.96	& 	0.87	& 	----	& 	----	&	----	\\
		&	Spectral		& 	0.95	&	0.30	& 	0.09	& 	----	& 	----	&	----	\\
		\midrule
		\multirow{6}{*}{\begin{tabular}{@{}l@{}}Time \\ (sec.)\end{tabular}}
		&	Our method	&	0.55	&	59	&	63	&	19	&	135	&	1284	\\
		&	Hu et al.	&	0.80	&	167	&	206	&	42	&	152	&	39196	\\
		&	Hierarchical	&	0.55	&	16	&	6	&	44	&	3066	&	9437	\\
		&	Louvain		&	0.38	&	9	&	6	&	14	&	89	&	520	\\
		&	Spectral		&	0.87	&	301	&	1855	&	24	&	265	&	1804	\\
		\bottomrule 
	\end{tabular}
	\caption{
		Results on six data sets. Our method generally does better than that of Newman and Hu et al. It is also notable that the choice of metric matters. For instance, on the MNIST example, the Louvain method gets a worse modularity score than our method but better agreement with the ground truth labels. Conversely, our method gets a lower modularity score than Louvain on the Plume40 example, but the segmentation our method produced for Figure~\ref{plume_images} more closely agrees with domain experts' knowledge of how the plume really looks. See Table~\ref{ssl_table} for an example of how a small amount of supervision with our method reduces this ambiguity. Dashes denote missing entries in cases where metadata was not available. The LFR50k example illustrates the ability of our approach to deal with a large number of small communities using recursive partitioning.
	}
	\label{six_table}
	\setlength\tabcolsep{6pt} 
\end{table}
\begin{table}
\setlength\tabcolsep{5.5pt} 
\begin{tabular}{l|llllllll}
\toprule
&			&{Jas. Rid.} &	{Samson}&{Cuprite}	&	{FLC}	&{Pavia U}	&{Salinas}	&	{Salinas 1}\\
\midrule                                                                                 
&	Nodes		&	19,800	&	14,820	&	30,162	&	208,780	&	207,400	&	7,092	&	111,063	\\
&	Edges		&$1.1*10^5$	&$8.3*10^4$	&$1.6*10^5$	&$1.5*10^6$	&$1.6*10^6$	&$4.7*10^5$	&$8.4*10^5$	\\
&	Communities 	& 	4	&	3	&	12	&	3	&	9	&	6	&	16	\\
\midrule                                                                                
\multirow{6}{*}{Modularity}                                                              
&	Our method	&	0.99	&	0.98	&	0.99	&	0.94	&	0.93	&	0.97	&	0.96	\\
&	Hu et al.	&	0.99	&	0.98	&	0.90	&	0.94	&	0.94	&	0.97	&	0.96	\\
&	Hierarchical	&	0.98	&	0.98	&	0.99	&	0.93	&	0.93	&	0.97	&	0.96	\\
&	Louvain		&	0.99	&	0.98	&	0.99	&	0.90	&	0.88	&	0.95	&	0.95	\\
&	Spectral		&	0.91	&	0.90	&	0.91	&	0.90	&	0.90	&	0.96	&	0.90	\\
&	Ref		&	0.90	&	0.90	&	0.90	&	0.90	&	0.90	&	0.90	&	0.90	\\
\midrule
\multirow{6}{*}{\begin{tabular}{@{}l@{}}Time \\ (sec.)\end{tabular}}
&	Our method	&	17	&	13	&	42	&	121	&	160	&	4.6	&	96	\\
&	Hu et al.	&	40	&	27	&	63	&	203	&	270	&	3.3	&	117	\\
&	Hierarchical	&	1.5	&	1.1	&	2.4	&	378	&	411	&	0.74	&	66	\\
&	Louvain		&	1.5	&	1.2	&	2.7	&	39	&	40	&	0.75	&	15	\\
&	Spectral		&	28	&	10	&	148	&	38	&	65	&	6.5	&	24	\\
\bottomrule 
\end{tabular}
\caption{Results on additional hyperspectral data sets. The resolution parameter was $0.1$, and the reference partition has all nodes in the same community. Our method achieves a top modularity score in each network except for Salinas, where Hu et al.'s method gets slightly higher results. Our method partitioned recursively and initialized with kmeans clustering on leading eigenvectors that had been computed for use in the pseudospectral scheme.}
\label{hsi_table}
\setlength\tabcolsep{6pt} 
\end{table}
\begin{table}
	\centering
\setlength\tabcolsep{5.5pt} 
	\begin{tabular}{l|lll}
		\toprule
		&			&	{Moons}	&	{MNIST}	\\
		\midrule                                                
		\multirow{3}{*}{Modularity}                             
		& 	Unsupervised 	& 	0.84	&    	0.91 	\\
		& 	10\% supervised	& 	0.84	&    	0.92 	\\
		& 	Reference 	&	0.83 	&    	0.92  	\\
		\midrule                                                
		\multirow{2}{*}{Classification}                                 
		&	Unsupervised 	&	0.97	&	0.90	\\
		&	10\% supervised	&	0.97	&	0.97	\\
		\midrule
		\multirow{2}{*}{Modularity Consistency}                 
		&	Unsupervised 	&	0.75	&	0.65	\\
		&	10\% supervised &	1.00	&	1.00	\\
		\midrule
		\multirow{2}{*}{Classification Consistency}                 
		&	Unsupervised 	&	0.75	&	0.05	\\
		&	10\% supervised &	1.00	&	0.65	\\
		\bottomrule 
	\end{tabular}
	\caption{
		Results of our method using networks constructed from the two moons and MNIST examples with and without 10\% supervision. Consistency here denoted percent of cases for which the results were within 2\% of the best value achieved. 
		In the two moons example, supervision improves consistent matching to metadata. In the MNIST example, both consistency and peak metadata matching are substantially improved. Note that in both cases, the peak modularity is not changed, indicating that the supervision helps the solver find local maxima that are more relevant to the classification task, thus addressing the well-known degeneracy issues of modularity's energy landscape.
		The code was run 20 times on each example.
	}
	\label{ssl_table}
	\setlength\tabcolsep{6pt} 
\end{table}

\subsection{Analysis of each experiment}

We now describe the individual experiments.

\begin{figure}
	\centering
	\includegraphics[width=.5\textwidth]{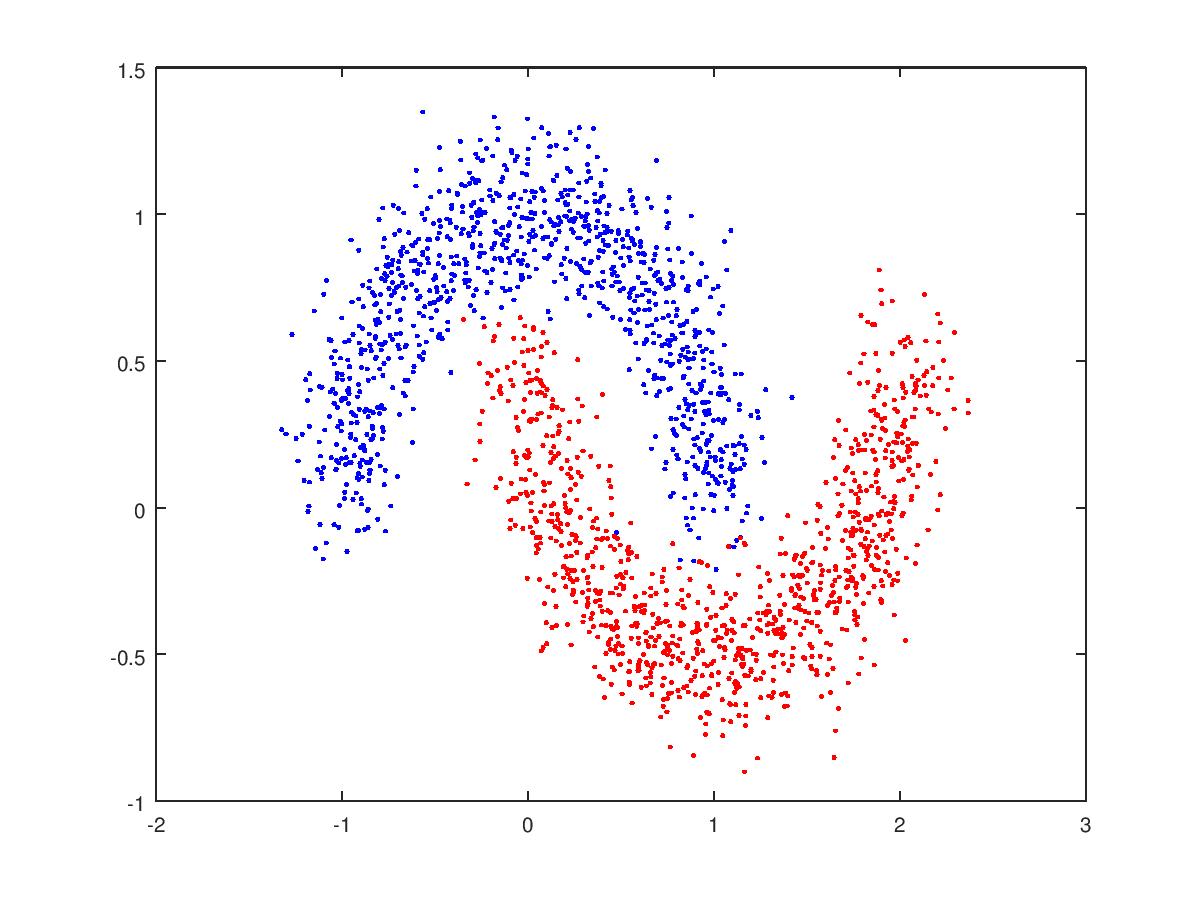}
	\caption{Projection of the two moons example onto two dimensions}
	\label{tm_fig}
\end{figure}
{\bf Two Moons} Two moons consists of 2,000 points in 100-dimensional space, sampled from two half-circles, with Gaussian noise added, see Figure~\ref{tm_fig}. We constructed a 13-nearest neighbors graph with the edge weights given by a Gaussian law, with locally-determined decay parameters \cite{zelnik_manor_perona}. The number of classes was assumed known, where the class of a point is the half-circle to which it originally belonged.

{\bf MNIST} MNIST consists of 70,000 28x28-pixel images, each of which contains a single handwritten digit~\cite{lecun_cortes_burges_mnist}. The task is to identify the digit in each image. The graph was constructed by projecting onto 50 principle components for each image and then using a 10-nearest neighbors graph with self-tuning Gaussian decay~\cite{zelnik_manor_perona}. The number of classes was assumed known. As in~\cite{hui}, 11 classes were used, as there are two different ways to write the digit $1$, with or without the top flag and flat base. This modularity landscape was particularly troublesome, with about $25\%$ of the partitions we found having better modularity than the ground truth partition, despite the fact that partitions with a classification accuracy greater than $95\%$ were found only about $4\%$ of the time. 

{\bf LFR 50k} This is a well-known ensemble of artificial networks~\cite{LFR_2008}. We used the following parameters to generate it: average degree of 20, maximum degree of 50, degree distribution exponent of 2, community size distribution exponent of 1, effective mixing parameter of 0.2, maximum community size of 50, minimum community size of 10. The large number of small communities makes this a challenging problem---similar experiments on a 1,000-node networks with 40 communities gave near-perfect classification. We use purity to gauge classification accuracy. Given two partitions $g_1$ and $g_2$, the purity is defined as $\frac{1}{N} \sum_{\alpha=1}^{\nhat} \max_{\beta=1,\ldots,\nhat} \#\{ i : g_1=\alpha \text{ and } g_2=\beta\}$, where $\#$ denotes the cardinality.

{\bf Urban Image} The urban hyperspectral image is a $307\times307$ image of an urban setting, where each pixel encodes the intensity of light at 129 different wavelengths. The classification problem is to identify pixels that contain similar materials, such as dirt, road, grass, etc. 

The graph representation was computed using ``nonlocal means''~\cite{buades_2005}, which means that for each pixel $p$, a vector $v_p$ was constructed by concatenating the data in a $3\times3$ window centered at $p$. One then uses a weighted cosine distance on these $3\times3\times 162 = 1,458$ component vectors, where the components from the center of the window are given the most weight. For each pixel, we obtained the 10 nearest neighbors in this distance using a k-d tree and the VLFeat software package~\cite{vlfeat}. The images in~\cref{wz_images} were selected from a collection of 200 segmentations as being the most visually appealing. We compared with a recent NLTV-based algorithm~\cite{wei_zhu}, which is specifically designed for hyperspectral imaging applications and found our segmentation competitive. We also compared with Modularity MBO and GenLouvain~\cite{GenLouvain} segmentations. For instance, Balanced TV does well at placing the grass into a single class and correctly resolved the difference between pavement and dirt. Balanced TV gives the sharpest resolution of the roads and the surrounding dirt in the upper right. Our method does have a little trouble compared to GenLouvain when resolving the buildings just below the large road in the upper left corner of the picture, although this is partly due to the fact that the roofs there are made of different materials from most of the houses further down in the image, and NLTV has a similar problem.

\begin{figure}
	\centering
	\begin{tabular}{cc}
		\begin{tabular}{c}
			\includegraphics[width=2in]{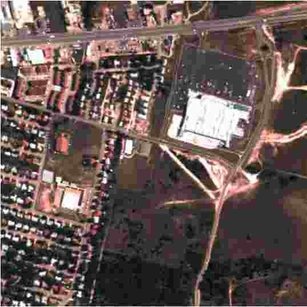}\\ 
			RGB image
		\end{tabular}
		&
		\begin{tabular}{c}
			\includegraphics[width=2in]{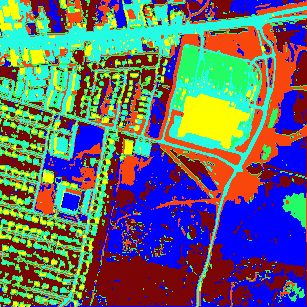}\\
			Our method
		\end{tabular}
		\\
		\begin{tabular}{c}
			\includegraphics[width=2in]{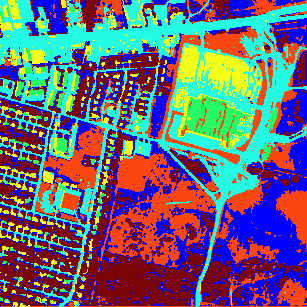}\\
			Modularity MBO
		\end{tabular}
		&
		\begin{tabular}{c}
			\includegraphics[width=2in]{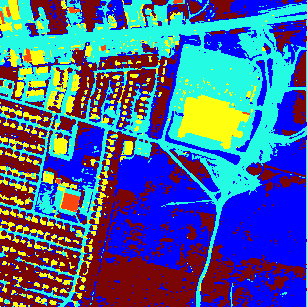} \\
			GenLouvain segmentation
		\end{tabular}
		\\
		\begin{tabular}{c}
			\includegraphics[width=2in]{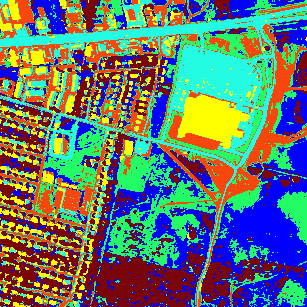}\\
			NLTV segmentation~\cite{wei_zhu}
		\end{tabular}
		&
	\end{tabular}
	\caption{
		The urban dataset segmented using different methods.
		Our method effectively separates the dirt from roads, resolving the roads in the upper right corner, and placing all of the grass into a single class. It has some difficulty with the buildings in the upper left corner, just below the main road, which are a different material from the other buildings.
	}
	\label{wz_images}
\end{figure}

{\bf Plume Hyperspectral Video} The gas plume hyperspectral video records a gas plume being released at the Dugway Proving Ground~\cite{gerhart2013detection,plumes_source,merkurjev_hyperspectral_2014}.\footnote{In~\cite{merkurjev_hyperspectral_2014}, a semi-supervised MBO-type approach was used.} The graph was constructed by the same procedure as the urban dataset, simply concatenating each frame side-by-side into one large image and using nonlocal means to form the graph. Each frame has $320\times 128$ pixels with data from $129$ wavelengths. Two versions of this dataset were used, one with 7 frames, and another with 40 frames. We have included the segmentation of one frame in Figure~\ref{plume_images}, together with segmentations produced by competing algorithms. Our method is the only one that places the entire plume in a single class. The images shown were chosen as the best out of thirty for visual appeal. 
\begin{figure}[H]
	\centering
	\begin{tabular}{cc}
		\includegraphics[width=2.8in]{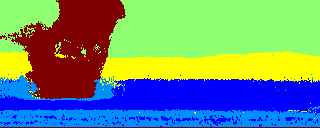} &
		\includegraphics[width=2.8in]{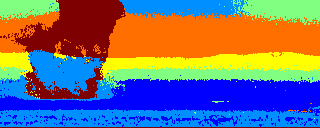} \\
		\small Our method & Spectral Clustering \\
		\includegraphics[width=2.8in]{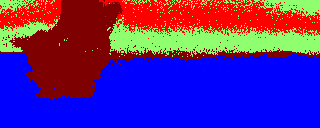} &
		\includegraphics[width=2.8in]{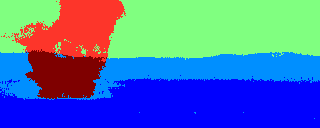} \\
		\small NLTV~\cite{wei_zhu} & GenLouvain
	\end{tabular}
	\caption{
		Segmentations of the plume hyperspectral video using different methods.
		Observe that our method is the only method that gets the whole plume into a single class without any erroneous additions.
	}
	\label{plume_images}
\end{figure}

{\bf Other Hyperspectral Examples} We included seven additional hyperspectral image examples, which are well-known in the image processing community. In each case, we formed the k-nearest neighbor graph using nonlocal means and VLFeat. See Appendix~\ref{sec:data} for more details. Overall, our algorithm performs very competitively on these examples in terms of modularity. The speed is slower than Louvain, but the run time is still very reasonable, and the modularity scores are more consistently good.

\section{Conclusion}
\label{sec:conclusion}

We have shown that modularity optimization can be framed as a balanced TV problem that is convex except for a discrete constraint. This formulation yields an energy landscape that is easier to understand by using terms with a ready intuitive meaning and by putting all of the nonconvexity into a simple discrete constraint. We have given a rigorous nonconvexity result and shown how to use the Ginzburg-Landau functional to approximate modularity optimization by more convex problems. We have also proposed an improved modularity optimization scheme, Balanced TV, which works very well even on large graphs and which requires much less hand-tuning. Numerical tests show that our method is competitive in terms of accuracy, while being faster than its predecessor, Modularity MBO.


\section*{Acknowledgements}
A. L. Bertozzi and Z. M. Boyd were supported by NSF grants DMS-1417674 and DMS-1118971.
X. C. Tai and Z. M. Boyd were supported by ISP-Matematikk (Project no.\ 2390033/F20) at the University of Bergen.
X. C. Tai was additionally supported by the startup grant at Hong Kong Baptist University.
Z. M.  Boyd was additionally supported by the U.S. Department of Defense (DoD) through the National Defense Science \& Engineering Graduate Fellowship (NDSEG) Program.

\begin{appendices}
	\section{Gamma convergence}
	\label{sec:gamma}
	The following are some basic facts about Gamma-convergence to aid in understanding the results of this paper. See~\cite{van_gennip_2012} for more details.
	\begin{definition}
		Let $X$ be a topological space, and $\mathcal{F}_n$ a sequence of real-valued functionals of $X$. Then the sequence is said to $\Gamma$-converge to a functional $\mathcal{F}$ on $X$ if the following two conditions hold:
		\begin{enumerate}
			\item For convergent sequence $x_n\to x$, we have $\liminf_{n\to\infty} \mathcal{F}_n(x_n) \leq \mathcal{F}(x)$.
			\item For every $x$, there exists a convergent sequence $x_n\to x$ such that $\limsup_{n\to\infty} \mathcal{F}_n(x_n) \ge F(x)$.
		\end{enumerate}
	\end{definition}
	For our purposes, $\Gamma$-convergence is primarily a tool for ensuring that the minimizers of $\mathcal{F}_n$ approach the minimizers of $\mathcal{F}$, as guaranteed by the following:
	\begin{theorem}
		Let $\mathcal{F}_n$ $\Gamma$-converge to $\mathcal{F}$, and let $x_n$ be a minimizer of $\mathcal{F}_n$. Then every cluster point of the $x_n$ is a minimizer of $\mathcal{F}$. If $\mathcal{G}$ is continuous, then $\mathcal{F}_n + \mathcal{G}$ $\Gamma$-converges to $\mathcal{F} + \mathcal{G}$.
	\end{theorem}

	We end with the proof of~\cref{thm:GL}.
	\begin{proof}
		We largely follow~\cite{van_gennip_2012}, generalizing and filling in a minor hole from that proof.

		Observe that all of the terms not involving the potential are continuous and independent of $\epsilon$, so they cannot interfere with the $\Gamma$-convergence~\cite{dal_maso}.
		Therefore, it suffices to prove that $\frac{1}{\epsilon}T$ $\Gamma$-converges to
		\[
			\chi(u) =
			\begin{cases}
				0 & \text{if $u$ corresponds to a partition} \\
				+\infty & \text{otherwise}.
			\end{cases}
		\]

		To prove the lower bound, let $u_n \to u$ and $\epsilon_n\to 0$. If $u$ corresponds to a partition, then $\chi(u)=0$, which is automatically less than or equal to $\frac{1}{\epsilon_n}T(u_n)$ for each $n$. If $u$ does not correspond to a partition, then $\chi(u) = +\infty.$ Pick $N_1$ such that whenever $n>N_1$, the distance from $u_n$ to the nearest feasible point is at least $c>0$. Letting $T_c$ be the infimum of $T$ on all of $\R^{N\times\nhat}$ minus the balls of radius $c$ surrounding each feasible point (so $T_0>0$ in particular). Then we have $\lim\inf_{n\to\infty} \frac{1}{\epsilon_n}T(u_n) \ge \lim_{n\to\infty} \frac{1}{\epsilon_n}T_0 = +\infty$. Thus, the lower bound always holds.

		To prove the upper bound, let $u$ be any $N\times\nhat$ matrix. If $u$ corresponds to a partition, then letting $u_n=u$ for all $n$ gives the required sequence. If $u$ does not correspond to a partition, then $u_n=u$ for all $n$ still satisfies the upper bound requirement.

		Thus both the upper and lower bound requirements hold, and we have proved $\Gamma$-convergence.
	\end{proof}

	\section{Deferred proofs}

	In this section, we give proofs of propositions stated earlier in the paper.

	\begin{proof}[Proof of~\Cref{prop:freezing}]
		We first get pointwise estimates on $u-u_0$:
		\begin{align}
			|| u-u_0||_\infty
			\le ||e^{-\tau M} - I ||_\infty ||u_0||_\infty
			=   ||e^{-\tau M} - I ||_\infty
			\le \sum_{k=1}^\infty \frac{1}{k!}\tau^k ||M||_\infty^k
			= e^{\tau ||M||_\infty} - 1 \label{eqn:estimate}
		\end{align}
		We estimate $||M||_{\infty}$ as follows:
		\begin{align*}
			||M||_{\infty} &= \max_i \sum_j |L_{ij} + \frac{\gamma}{m}k_ik_j|
			= \max_i \sum_j |k_{i}\delta_{ij} - w_{ij} + \frac{\gamma}{m}k_ik_j|
			\\&\le \max_i k_{i} + k_{i} + \frac{\gamma}{m}k_i 2m
			= 2(1+\gamma) k_{\mathrm{max}}
		\end{align*}

		These computations do not depend on $\nhat$, but in order to get a timestep, we assume that $\nhat = 2.$ In this case, let $u^1$ and $u^2$ be the columns of $u$. We have
		$u^1_t = -Mu^1$ and $u^2_t = -Mu^2.$
		Subtracting these, and letting $v = u^1 - u^2$ yields $v_t = -Mv.$
		Allowing $v$ to evolve until the time of thresholding, we see that node $i$ will switch classes if and only if $v(i)$ has changed sign, that is if $|v-v_0|_i > 1.$ The quantity in~\cref{eqn:estimate} is less than $1$ exactly when $\tau < \frac{ \log 2 }{ 2(\gamma + 1) k_{\mathrm{max}} } \approx \frac{0.15}{(\gamma+1) k_{\mathrm{max}}}$. This is exactly the bound we sought.

		Next, we work on the $L^2$ bound
		\begin{align}
			|| u-u_0||_\infty &\le \sqrt{\nhat}|| u-u_0||_2 \le \sqrt{\nhat}||e^{-\tau M} - I ||_2 ||u_0||_2 \le\sqrt{\nhat} ||u_0||_2 \sum_{k=1}^\infty \frac{1}{k!}\tau^k ||M||_2^k
			\\&= \sqrt{\nhat}||u_0||_2 \left( e^{\tau ||M||_2} - 1 \right)
			= \sqrt{\nhat}||u_0||_2 \left( e^{\tau \rho} - 1 \right) \label{eqn:l2_estimate}
		\end{align}
		As before, when we let $\nhat=2$, one can subtract the columns to get $v$, so that no node will switch communities as long as $||v - v_0||_{\infty} < 1$, which is guaranteed if
		$\tau < \rho^{-1} \log\left( 1+N^{-\frac{1}{2}} \right).$
	\end{proof}
%
	\begin{proof}[Proof of~\cref{prop:steady_state}]
%
		To get the bound, we let $\Lambda$ be a diagonal matrix with the eigenvalues of $M$ on the diagonal. Since $M$ is positive semi-definite, we can write $M = Q\Lambda Q^T$ for some orthogonal matrix $Q$. Then we have
		\[
			||u(\tau)||_2 = ||e^{-\tau M} u_0 ||_2 \le ||e^{-\tau M}||_2 || u_0 ||_2 = ||e^{-\tau \Lambda}||_2 || u_0 ||_2 = e^{-\tau \lambda_1} || u_0 ||_2
		\]
		Setting the latter quantity less than $\epsilon$ and then solving for $\tau$ yields the required bound.
	\end{proof}
	\section{Hyperspectral Image Details}
	\label{sec:data}

	In this appendix we collect some basic facts about the images used in~\cref{hsi_table}.
	\begin{itemize}
		\item Jasper Ridge: An image of a river area. It has 198 channels and 100x100 pixels. Retrieved from \url{http://www.escience.cn/people/feiyunZHU/Dataset_GT.html}.
		\item Samson: An image of a coastline. It has 156 channels and 952x952 pixels. Retrieved from \url{http://www.escience.cn/people/feiyunZHU/Dataset_GT.html}.
		\item Cuprite: An image of ground near Las Vegas. It has 224 channels and 250x190 pixels. Retrieved from \url{http://www.escience.cn/people/feiyunZHU/Dataset_GT.html}.
		\item FLC: A moderate-dimensional image. It has 12 channels and 949x220 pixels. Available at \url{ftp://www.daba.lv/pub/TIS/atteelu_analiize/MultiSpec/tutorial/ModDimensionDataSet.zip}.
		\item Pavia U: An image of Pavia University in Northern Italy. It has 103 channels and 610x610 pixels. Retrieved from \url{http://lesun.weebly.com/hyperspectral-data-set.html}.
		\item Salinas: An image containing vineyard fields, soils, and vegetation. It has 224 channels and 512x217 pixels. Retrieved from \url{http://www.ehu.eus/ccwintco/index.php?title=Hyperspectral_Remote_Sensing_Scenes}.
		\item Salinas 1: A subimage of the previous image containing 86x83 pixels.
	\end{itemize}

\end{appendices}

\nocite{*}
\bibliographystyle{siamplain}
\bibliography{ConvexModularity}

\end{document}